\documentclass[12pt,leqno,amscd,amssymb,verbatim, url]{amsart}
\usepackage{amsfonts,amssymb,graphicx}
\usepackage{amsmath,amscd}
\usepackage[english]{babel}
\usepackage{tikz}
\usetikzlibrary{arrows,shapes}\oddsidemargin .2in \evensidemargin .2in \textwidth 6.1in
\newtheorem{thm}{Theorem}[section]
\usepackage{amsfonts,latexsym, curves}
\newtheorem{lem}[thm]{Lemma}
\newtheorem{cor}[thm]{Corollary}

\newtheorem{prop}[thm]{Proposition}

\theoremstyle{definition}

\newtheorem{rem}[thm]{Remark}
\newcommand{\bpict}{\begin{picture}}

\newcommand{\epict}{\end{picture}}

\numberwithin{equation}{thm}

\newcommand{\gr}{\text{\rm gr}}

\newcommand{\sC}{{\mathcal {C}}}

\newcommand{\Ext}{{\text{\rm Ext}}}

\newcommand{\fa}{{\mathfrak a}}

\newcommand{\Hom}{\text{\rm Hom}}

\newcommand{\End}{\operatorname{End}}

\newcommand{\ind}{\operatorname{ind}}

\newcommand{\soc}{\operatorname{soc}}

\newcommand{\rad}{\operatorname{rad}}

\newcommand{\Ker}{\operatorname{Ker}}

\newcommand{\wQ}{{\widetilde{Q}}}

\newcommand{\Dist}{{\text{\rm Dist}}}

\newcommand{\Gmod}{G{\text{\rm --mod}}}

\newcommand{\blist}{\begin{list}{\rom{(\roman{enumi})}}{\setlength
{\leftmarg in}{0em} \setlength{\itemindent}{7ex}
\setlength{\labelsep}{2ex}\setlength{\listparindent}{\parindent}
\usecounter{enumi}}}
\newcommand{\elist}{\end{list}}

\thanks{Research supported in part by the National Science
Foundation}

 \dedicatory{We dedicate this paper to the memory of Daya-Nand Verma.}

\begin{document}
 \title[Forced gradings and the Humphreys-Verma conjecture]{Forced gradings and the Humphreys-Verma conjecture}

 \author{Brian J. Parshall}\address{Department of Mathematics \\
University of Virginia\\
Charlottesville, VA 22903} \email{bjp8w@virginia.edu {\text{\rm
(Parshall)}}}
\author{Leonard L. Scott}
\address{Department of Mathematics \\
University of Virginia\\
Charlottesville, VA 22903} \email{lls2l@virginia.edu {\text{\rm
(Scott)}}}

\maketitle

\begin{abstract}Let $G$ be a semisimple, simply connected algebraic group defined and split over
a prime field ${\mathbb F}_p$ of positive characteristic. For a positive integer $r$, let $G_r$ be the
$r$th Frobenius kernel of $G$. Let $Q$ be a projective indecomposable (rational) $G_r$-module. The
well-known Humprheys-Verma conjecture (cf. \cite{Ballard}) asserts that the $G_r$-action on $Q$ lifts to an rational
action of $G$ on $Q$. For $p\geq 2h-2$ (where $h$ is the Coxeter number of $G$), this conjecture was proved
by Jantzen \cite{Jan0} in 1980, improving Ballard \cite{Ballard}. However, it remains open for general characteristics. In this paper, the
authors establish several graded analogues of the Humphreys-Verma conjecture, valid for all $p$. The most general of our results, proved in full here, was announced (without proof) in \cite{PS2}. 
Another result relates the Humphreys-Verma conjecture to earlier work of  Alperin, Collins, and Sibley 
on finite group representation theory; see \cite{ACS}. A key idea in all formulations involves the notion of
a forced grading. The latter goes back, in particular, to the recent work \cite{PS3}, \cite{PS4}, relating graded structures and $p$-filtrations.  The authors anticipate that
the Humphreys-Verma conjecture results here will lead to extensions to smaller characteristics of these earlier papers \cite{PS3}, \cite{PS4}.

\end{abstract}

\section{Introduction}Let $G$ be a semisimple, simply connected algebraic group defined and split
over the prime field ${\mathbb F}_p$ for some prime $p$. If $F:G\to G$ is the Frobenius morphism and
$r$ is a positive integer,
let $G_r=\ker  F^r$ be the $r$th Frobenius kernel of $G$. Thus, $G_r$
is an infinitesimal subgroup in the sense of \cite[I.8.1]{Jan} (see also below). The representation theory of $G_r$ is an important ingredient
in the representation theory of $G$ and its finite subgroups of Lie type. For any $r$-restricted
dominant weight $\lambda$, let $Q_r(\lambda)$ be the projective cover in the category $G_r$-mod
of rational, finite dimensional $G_r$-modules. According to Jantzen (in 2003), the issue as to whether $Q_r(\lambda)$ lifts to a rational $G$-module
is ``one of the main problems in representation theory of Frobenius kernels \dots [and] has been
open for about 25 years." See \cite[p. 526]{Jan}. The origins of this
question go back even earlier to important work of Humphreys and Verma \cite{HV} in 1973. In fact,
for  $r=1$, they stated the extendibility to $G$ as an ``unproved" theorem; see also Humphreys \cite{Hump76} and its postscript. This
unproved result became known as the ``Humphreys-Verma conjecture" in Ballard's paper \cite{Ballard}, where it was proved for
$p\geq 3h-3$ ($h=$ Coxeter number). This bound was improved by Jantzen \cite{Jan0} to  $p\geq 2h-2$. For
smaller primes, the conjecture remains open today. In 1993, Donkin \cite{Donkin2}
proposed the stronger conjecture (in all characteristics $p$) that each $Q_r(\lambda)$ is
the restriction to $G_r$ of an appropriate $G$-tilting module. He proved this conjecture if  $p\geq 2h-2$. 

A number of general consequences of the Humphreys-Verma conjecture are discussed in \cite[pp. 334--341]{Jan}. More recently, the conjecture has played an important role in  work of the authors of this paper.
For example, a  weaker, ``stable" version of the Donkin conjecture was proved in \cite{PS1},  an
essential step in obtaining results  on  bounding cohomology for semisimple algebraic groups.  In a different direction,
the arguments in \cite{PS4}  on filtrations of Weyl modules by $p$-Weyl modules required  the Humphreys-Verma conjecture, leading to the blanket assumption there that $p\geq 2h-2$. A similar assumption was
required in the precursor \cite{PS3}.

The present paper establishes a version of the Humphreys-Verma conjecture in a ``forced grading" setting.
The terminology refers to first thinking of representations in terms of modules for algebras, then passing
to graded versions of the latter obtained from filtrations by natural series of ideals. In \cite{PS3} and \cite{PS4}, 
such a ``forced grading" approach was necessary, so the present paper opens up the
possibility of recovering the conclusions of these papers (as well as other applications) with weaker 
hypotheses on $p$. We expect to pursue this later. 

A novel by-product of this investigation is an observed analogy of (the forced graded version of) the Humphreys-Verma conjecture
with a finite group representation theoretic result of Alperin, Collins, and Sibley \cite{ACS}. This latter work is
recast in order to bring out this analogy. This approach leads to an improved statement of the theorem, and
there are fruitful consequences on the algebraic group side as well; see below.

This paper stems from the announcement \cite[footnote 11]{PS2} by the authors (without proof) of a result that
contains our more focused work on the Humphreys-Verma conjecture in \S3 as a special case. A proof of this more
general result is provided in the final section of the present paper.

In the remainder  of this introduction, we give a few more details of these results.

Let $G$ be an affine algebraic group, and let $N$ be a closed normal subgroup. We allow the possibility that $N$ could be an infinitesimal subgroup, for example, a Frobenius kernel. If $\rho:N\to GL(Q)$
is a (finite dimensional) rational representation, for $g\in G$, define $\rho^g:N\to GL(Q)$ to be $\rho^g(n):=
\rho(gng^{-1})$. Since $N$ might be infinitesimal in our context, we should really work at the level
of $S$-points, for commutative $k$-algebras $S$. In any event, $\rho^g$ is a rational representation of
$N$, and a necessary condition that the action of $N$ extend to an action of $G$ is that
$\rho^g$ be equivalent to $\rho$, for all $g\in G$. Equivalently, there exists a function $\alpha:G\to
GL(Q)$ such that $\alpha(g)\rho(n) =\rho^g(n)\alpha(g)$, for all $n\in N$. If $\alpha$ can be taken as
 a morphism of varieties, then $Q$ is {\it rationally} stable. When $N$ is a Frobenius kernel and $Q$
is an indecomposable projective module, then $Q$ contains an $N$-submodule $V$ that lifts to $G$ (and contains the $N$-socle of $Q$).
This led to the notion of {\it strong} stability in \cite{PS2}, which is reviewed in \S2, along with a practical
way Lemma \ref{biglemma} of verifying it. The latter result and the approach of this section  were motivated
by Donkin \cite{Donkin1}, as discussed in \cite{PS2}. In general, when
$(Q,V)$ is a strongly stable pair, we can construct a Schreier extension $G^\diamond$ of $G$ by a
unipotent subgroup $U$. The group $G^\diamond$ is an affine algebraic group, and $Q$ is a rational $G^\diamond$-module. 
The action of $N$ on $Q$ extends to a rational action of $G$ (and agreeing with the action of $G$ on $V$)
if and only if the exact sequence $1\to U\to G^\diamond\to G\to 1$ is split by a homomorphism that agrees
with the obvious inclusion $N\subseteq G^\diamond$; see \cite[\S3]{PS2}. 

Now assume that $N$ is infinitesimal subgroup (so that $k[N]$ is a finite dimensional algebra and $N$ has only one $k$-point).
Let $\Dist(G)$ (resp., $\Dist(N)$) be the distribution algebra of $G$ (resp. $N$). We can grade $\Dist(G)$
by its $\Dist(N)$-radical series, namely, we form the algebra
\begin{equation}\label{gradedalgebra}
\gr_N\Dist(G):=\bigoplus_{n\geq 0}( \rad \Dist(N))^n\Dist(G)/(\rad \Dist(N))^{n+1}\Dist(G),
\end{equation}
as well as the $\gr_NQ$-module
\begin{equation}\label{gradedmodule}
\gr_NQ:=\bigoplus_{n\geq 0}( \rad \Dist(N))^nQ)/(\rad \Dist(N))^{n+1}Q.
\end{equation}
Observe that $\rad\Dist( N)$ is a finite dimensional nilpotent ideal in $\Dist(N)$. Our first main result, given in
Theorem \ref{firstFrobenius}, proves that $\gr_NQ$ is a $\gr_N\Dist(G)$-module with an action that extends
the natural action of $\gr_N\Dist(N)$ on $\gr_NQ$. Here, $\gr_N\Dist(N)$ is defined by replacing
$\Dist(G)$ in (\ref{gradedalgebra}) by $\Dist(N)$.  In this discussion, the projective module $Q$ can, in
fact, be taken to be an arbitrary finite dimensional rational $N$-module, when $(Q,V)$ is a strongly stable pair; see
Theorem \ref{maintheorem}.

The distribution algebra of an abstract finite group $G$ is trivial. However, it is possible to replace it by
the group algebra $kG$.\footnote{More generally, if $G$ is a finite group scheme, it has a ``measure algebra"
$M(G)$ which specializes to the distribution algebra in the infinitesimal case and to the group algebra 
in the (discrete) finite group case.  In general, $\Dist(G)\subseteq M(G)$ and $M(G)$-mod always identifies with
$\Gmod$; see \cite[I.8]{Jan}. It is likely that the results of \S4 carry over to general finite group schemes
using $M(G)$,
but we have not pursued this. Readers interested in this direction in \S4 (or in generalizing \S5 to disconnected groups) should consult Remark 
4.6(d) on Clifford systems.} In this context, there is an analogous version of Theorem \ref{maintheorem} proved in
Theorem \ref{Generalizationtofinitegroups}. Making use of this result, we prove in Theorem \ref{mainfinitetheorem} the analogue mentioned above of a result of Alperin, Collins, and Sibley \cite{ACS}.
Also, Remark \ref{lastremark}(c) indicates how a character-theoretic result of Donkin in the algebraic group case
can be cast in the Alperin, Collins, and Sibley framework.

Finally, \S5 takes up the general, more difficult case in which $N$ is no longer assumed to
be infinitesimal. In this case,  definition (\ref{gradedalgebra}) does not make sense. Instead, it is
necessary to replace $\Dist(G)$ by a finite dimensional quotient $\Dist(G)/I$. In this context, a
variant of Theorem \ref{maintheorem} is obtained in Theorem \ref{footnotetheorem}. This is the result 
mentioned above that was
 announced without proof in \cite[footnote 11]{PS2}.
 
\section{Preliminaries}

We begin with a brief review of some basic results and notation from \cite{PS2} which
will be required in the sequel. Fix the algebraically closed field $k$. Let $G$ be
an affine algebraic group over $k$,\footnote{This means that $G$ is a reduced, affine algebraic group scheme
over $k$.} and let $\Gmod$ be the category of finite dimensional rational $G$-modules. Let $N$ be a closed subgroup scheme of $G$. Suppose that $\rho:N\to GL(Q)$
defines a finite dimensional rational representation of $N$. Also, let $V$ be a rational $G$-module such that
the restriction $V_N:=V\downarrow N$ of $V$ to $N$ appears as an $N$-submodule of $Q$. For $g\in G$, let $\rho^g:N\to GL(Q)$ be the rational representation of $N$ defined by twisting by $g$, namely, $\rho^g:N\to GL(Q)$, $n\mapsto\rho(gng^{-1})$,
for $n\in N$. The pair $(Q,V)$ is  strongly $G$-stable (with respect to $N$) provided there is a morphism $\alpha:G\to GL(Q)$ of varieties
such that the following conditions are satisfied:

\begin{itemize}
\item[(1)] $\alpha(g)\rho(n)=\rho^g(n)\alpha(g)$; $\alpha(1)=1_Q;$
\item[(2)] $\alpha(gn)=\alpha(g)\rho(n)$;
\item[(3)]
$\alpha(g)v=gv$,\end{itemize}
for all $g\in G, n\in N, v\in V$. (To be precise, since $N$ is not assumed to be reduced, we should
work at the level of $S$-points for all commutative $k$-algebras
$S$. See \cite{PS2} for a more detailed discussion; we will continue with this convention below.)

Assume that $V$ contains the $N$-socle of $Q$. By \cite[Lem. 2.1]{PS2}, the annihilator $J_V$ in $\End_N(Q)$ of $V$  is a nilpotent ideal in
$\End_N(V)$, so that $U:=1_Q+J_V$ is a unipotent subgroup of $GL(Q)$ commuting with $N$ ``elementwise." There
is also a conjugation action $\kappa:G\times U\to U$ of $G$ on $U$, setting $\kappa(g,u)={^gu}:=\alpha(g)u\alpha(g)^{-1}$ for $g\in G$, $u\in U$.
Define $\gamma:G\times G\to U$ by putting $\gamma(g,h):=\alpha(g)\alpha(h)\alpha(gh)^{-1}$, $\forall
g,h\in G$. Then the pair $(\kappa,\gamma)$ is a Schreier system (in the sense of Hall \cite[\S15.1]{Hall};
see also \cite[\S2]{PS2} for the scheme-theoretic version). These data define a natural extension $G^\diamond$ of $G$ by $U$, so that
there is a short exact sequence 
\begin{equation}\label{diamond}1\to U\longrightarrow G^\diamond \overset\pi\longrightarrow G\to 1.\end{equation}
As a set (and scheme), $G^\diamond=U\times G$, with set-theoretic (or scheme-theoretic) section $\iota:G\to G^\diamond$, $g\mapsto (1,g)$.
Also, $\iota$ maps $N$ isomorphically onto a subgroup (denoted $N^\diamond$) of $G^\diamond$, which
commutes element wise with $U$ and satisfies $N\cap H=1$. 

In addition, the action of $N$ on $Q$ extends to a rational action of $G$ on $Q$, agreeing with the action
of $G$ on $V$, if and only if the exact sequence (\ref{diamond}) is split by an algebraic group morphism
$G\to G^\diamond$ extending the isomorphism $N\overset\sim\longrightarrow N^\diamond$.  In
turn, this statement is equivalent to requiring that the exact sequence $1\to U\to G^\diamond/N^\diamond
\to G/N\to 1$ is split. See \cite[Cor. 3.7]{PS2}.

The following result, proved in \cite{PS2} and motivated by Donkin \cite{Donkin1}, provides a way to obtain strongly stable pairs. 

\begin{lem}\label{biglemma}(\cite[Lem. 3.1]{PS2}) Let $G$ be an affine algebraic group with closed, normal subgroup scheme $N$. Let $Q$ be a finite dimensional, rational $N$-module. Assume that there exists a rational $G$-module $M$ such that $M|_N\cong Q\oplus R$ for some $R$ in $N$-mod, and that there is a $G$-submodule
$V$ of $M$ contained in $Q$ and containing $\soc^NQ$. Then the pair $(Q,V)$ is strongly $G$-stable.  
\end{lem}

In what follows, considerable use will be made of the distribution algebra $\Dist(G)$ of an algebraic group
scheme over $G$.   We do not define this notion here, but instead refer
to Chapters 7 and 8 of Jantzen's book \cite{Jan} for definitions and elementary results. We generally follow his notation. In particular, observe that $\Dist(G)$ has a natural Hopf algebra structure. A morphism $G\to H$  of algebraic group schemes over $k$ induces a morphism $\Dist(G)\to\Dist(H)$ of Hopf algebra, which is
an inclusion if $G$ is a closed subgroup scheme of $H$.  Also, a (finite dimensional) rational representation $\rho:G\to
GL(V)$ induces a natural module structure of $\Dist(G)$ on $V$ by means of an algebra homomosphism
$\Dist(\rho):\Dist(G)\to \End(V)$.  By
\cite[Lemma, I.7.16]{Jan},  if $G$ is a connected affine algebraic group, then the category $G$-mod of finite
dimensional rational $G$-modules
fully embeds into the category $\Dist(G)$-mod of finite dimensional $\Dist(G)$-modules. If, in addition, $G$
is a connected semisimple, simply connected group, this embedding is an equivalence of categories
\cite[p. 171]{Jan}. (In this case, $\Dist(G)\cong k\otimes {\mathcal U}_{\mathbb Z}$, where ${\mathcal U}_{\mathbb Z}$ is the Kostant $\mathbb Z$-form of the universal enveloping algebra of the complex
semisimple Lie algebra having the same root type as $G$, see \cite{Haboush}, \cite{CPS2}.) Also,
there is an equivalence $\Gmod\overset\sim\longrightarrow \Dist(G)$-mod, when $G$ is infinitesimal
\cite[p. 114]{Jan}.  

Now assume that $k=\overline{{\mathbb F}_p}$ for a prime $p$. Let $N$ be a normal, closed infinitesimal subgroup scheme of an affine algebraic group $G$. 
Regarding $\Dist(N)$ as a subalgebra of $\Dist(G)$, it is stable under the conjugation action of $G$ on
$\Dist(N)$, and therefore stable under the adjoint action of $\Dist(G)$
on itself.  In particular, the radical $\rad\Dist(N)$ must be stable under this adjoint action. That is, if $x\in\Dist(G), r\in \rad\Dist(N)$, then (in Sweedler notation) 
$$xr=\sum (x_1rx_2^*)x_3\in \rad\Dist(N)\cdot\Dist(G),$$
where $x_2^*$ denotes the image of $x_2$ under the antipode of $\Dist(G)$. 
Therefore, $$\Dist(G)(\rad\Dist(N))\subseteq (\rad\Dist(N))\Dist(G).$$
A similar argument gives the reverse containment. It follows that 
 the expression (\ref{gradedalgebra}) has a natural algebra structure.  More simply,
$\gr_N\Dist(N)$ carries a natural (finite dimensional) algebra structure, and an algebra homomorphism
$\gr_N\Dist(N)\to\gr_N\Dist(G)$. In this way, there is a restriction functor
$$\gr_N\Dist(G){\text{\rm --mod}}\longrightarrow\gr_N\Dist(N){\text{\rm --mod}}.$$
Stated another way, one can speak of a $\gr_N\Dist(N)$-module extending to $\gr_N\Dist(G)$ with an
action compatible with that of $\gr_N\Dist(N)$. 

The following result, while not required for the results later in the paper, is included for completeness.

\begin{prop} With the above notation, the algebra homomorphism $\gr_N\Dist(N)\to\gr_N\Dist(G)$
is injective.\end{prop}

\begin{proof} We only sketch a proof. Put $N_1:=N$. Construct a increasing sequence
$N_1\subseteq N_2\subseteq N_3\cdots$ of infinitesimal subgroups of $G$ so that $\Dist(G)=\bigcup_n\Dist(N_n)$.  For $n\geq 1$, we claim that $\gr_N\Dist(N)\subseteq\gr_N\Dist(N_n)$. In fact, $k[N]$ is
an injective rational $N$-module. Thus, $\Dist(N)$ is a projective, and hence injective, rational $N$-module.
This means that the inclusion $\Dist(N)\subseteq\Dist(N_n)$ of rational $N$-modules splits, which forces
$\gr_N\Dist(N)$ to be contained in $\gr_N\Dist(N_n)$. Finally, it is easy to check that if $0\not=x$ has
grade $n$ in $\gr_N\Dist(N)$ and maps to 0 in $\gr_N\Dist(G)$, then it must map to 0 in some
$\gr_N\Dist(N_n)$ for some $i\gg 0$. This is a contradiction.\end{proof}

\begin{rem} In case $G$ is semisimple, simply connected and $N=G_r$ is the $r$th Frobenius kernel, we
can take $N_i=G_{i+r-1}$ in the above proof. In this case, the proposition can also be proved by
modifying the standard basis of $\Dist(G)$ to prove that it is a free left $\Dist(G_r)$-module with
basis (in the standard notation of \cite[p. 156]{Hump}, reduced mod $p$) in which the $a_i,b_i,c_i$ are all non-negative
integers divisible by $p^r$.  

More generally, $G/N$ is defined over some finite field $\mathbb F$, so that a similar argument works
using pull backs of Frobenius kernels (in the see \cite[I.9.4]{Jan}) in $G/N$ defined over $\mathbb F$. 
 \end{rem}

\section{A forced graded version of the Humphreys-Verma conjecture}  For a prime $p$, let $G$ be a semisimple,
simply connected (and connected) algebraic group defined and split over the prime field ${\mathbb F}_p$.
(And let $k:=\overline{{\mathbb F}_p}$.)
Consider a Steinberg endomorphism $\sigma:G\to G$, i.~e., $\sigma$ is an algebraic group endomorphism
which has a finite
fixed-point subgroup $G(\sigma)$. Let $G_\sigma$ denote the scheme-theoretic kernel of $\sigma$. Then $G_\sigma$ is a closed, normal infinitesimal subgroup scheme of $G$. Let $F:G\to G$ be the Frobenius morphism over ${\mathbb F}_p$. Except when $G(\sigma)$
is a Ree or Suzuki group, $G_\sigma$ identifies with a Frobenius kernel $G_r:=\text{\rm Ker}(F^r)$
of some positive integer $r$. The distribution algebras of these Frobenius kernels are explicitly
described in \cite[Lemma, II.3.3]{Jan}.

  For convenience, let $N:=G_\sigma$, where $\sigma$ is as in the previous
 paragraph. Then $N$ is a normal, infinitesimal subgroup scheme of $G$. The stability theory
 of the previous section can be applied in this situation.

 \begin{thm}\label{maintheorem} Let $(Q,V)$ be a strongly stable pair (with respect to $N$).   Then the
 $\gr_N\Dist(N)$-module $\gr_NQ$ has a structure
 of a graded $\gr_N\Dist(G)$-module, agreeing with its natural $\gr_N\Dist(N)$-module structure.\end{thm}

 \begin{proof} We work with the algebraic group $G^\diamond$ discussed in \S1 above. 
As varieties, $G^\diamond= U\times G'$, where $G'$ is a subvariety of $G^\diamond$ which projects
isomorphically onto $G$ (but is generally not a subgroup of $G^\diamond$). Therefore,
 $$ \Dist(G^\diamond)\cong \Dist(U)\otimes \Dist(G').$$
 In addition, $N$ is a closed subgroup scheme of $G^\diamond$ when it is identified with $1\times N$.
 As such, it commutes with $U$. Thus, $\Dist(N)$ is a  of $\Dist(G^\diamond)$ (and of
 $\Dist(G')$)  commuting
 elementwise with $\Dist(U)$. We filter the relevant algebras by powers of the radical of $\Dist(N)$. Passing
 to graded algebras, this
 gives a short exact sequence
 $$0\to\Dist^+(U)\otimes\gr_N\Dist(G')\to\gr_N\Dist(G^\diamond)\to\gr_N\Dist(G)\to 0.$$
 Clearly, the algebra $\gr_N\Dist(G^\diamond)$ acts naturally on
 $$\gr_NQ:=\bigoplus_i\frac{(\rad\Dist(N))^i Q}{(\rad\Dist(N))^{i+1}Q}.$$
 The augmentation ideal  $\Dist^+(U)$ of $\Dist(U)$ is concentrated in degree 0 in $\gr_N\Dist(G^\diamond)$
 and so acts trivially on $\gr_NQ$. Thus, the action of $\gr_N\Dist(G^\diamond)$ on $\gr_NQ$ factors
 through an action of the algebra $\gr_N\Dist(G)$, as required.
  \end{proof}

As a corollary of the theorem (and Lemma \ref{biglemma}), we can now obtain the following forced graded version of the Humphreys-Verma conjecture. 

\begin{thm}\label{firstFrobenius} Assume that $G$ is a semisimple, simply connected algebraic group defined and
split over ${\mathbb F}_p$. Let $N=G_\sigma$ be the (group scheme theoretic) kernel of a Steinberg endomorphism $\sigma:G\to G$, and let $Q$ be a projective indecomposable module for $N$. Then
$\gr_NQ$ has the structure as a $\gr_N\Dist(G)$-module, compatible with the natural action of
$\gr_N\Dist(N)$ on $\gr_NQ$.\end{thm}

\begin{proof} It is enough to check the hypotheses in Lemma \ref{biglemma}. We can merely repeat an argument given in \cite{PS2}. The $N$-socle $V$ of $Q$ is an irreducible $N$-mdoule. Hence, $V$ extends to a rational $G$-module, which is necessarily irreducible.
Let $I$ be the injective envelope of $V$ in the category of rational $G$-modules. Because $I$ remains
injective as a rational $N$-module, the inclusion $V\hookrightarrow I$ extends to an inclusion $Q\hookrightarrow
I$ of rational $N$-modules. Thus, we can view $Q$ as an $N$-submodule of $I$, whose $N$-socle is
a $G$-submodule.. Now let $M$ be the
(necessarily finite dimensional) rational $G$-submodule of $I$ generated by $Q$. Because $Q$ is
an injective $N$-submodule of $M$, the hypotheses of Lemma \ref{biglemma} hold.
\end{proof}

\section{Variations of a theorem of Alperin-Collins-Sibley} In this section, $G$ denotes a
finite group (in the traditional sense, i.~e., a {\it reduced} finite group scheme). Let $N$ be a closed normal
subgroup of $G$ (which will be necessarily reduced). Let $k$ be a field of positive characteristic $p$.
In \cite[Theorem]{ACS}, Alperin, Collins, and Sibley
prove that there is a finite dimensional $kG$-module $M$ with the following
property: given any irreducible $kG/N$-module $S$, form its projective cover $Q$ (resp., $\widetilde Q$)
in the category of $kG$-modules (resp., $kG/N$-modules). Then $Q$ and $\widetilde Q\otimes M$
have the same composition factors. In this section, we will recast this result.

The proof of the following result is a straightforward modification of the proof of Theorem \ref{maintheorem}, replacing
the distribution algebras there by group algebras. We leave the proof of the reader. 

\begin{thm}\label{Generalizationtofinitegroups} Let $L$ be an irreducible $kN$-module which extends to a $kG$-module (still denoted $L$, though
its extension to $G$ may not be uniquely determined). Let $T=T(L)$
be the projective indecomposable cover of $L$ in $kN$-mod. Then $\gr_NT$ has the structure of
a graded $\gr_NkG$-module, extending its natural $\gr_NkN$-module structure. (Observe that
$\gr_NkN$ is a subalgebra of $\gr_NkG$.) Moreover, it can be assumed that $\gr_NT$ has
head $L$.
\end{thm}
 
 We next establish the
 following very general result on groups $G$ and finite normal subgroups.  
 
 \begin{lem}\label{gradedfactorization} Let $X$ be a graded $\gr_NkG$-module, and let $Y$ be a $kG/N$-module. Give $X\otimes Y$ the natural grading it inherits from $X$. Thus,
 $(X\otimes Y)_n:=X_n\otimes Y$ for all $n$.  Then there is a natural $\gr_NkG$-module structure on $X\otimes Y$
 in which 
 \begin{equation}\label{welldefined}
 [gu_n]_n(x_m\otimes y)=[gu_n]_nx_m\otimes gy\in X_{m+n}\otimes Y,\end{equation}
 for each $g\in G$, $u_n\in (\rad kN)^n$, $x_m\in X_m, y\in Y$, $n\in\mathbb N, m\in\mathbb Z$.
 In particular, $\gr_NkN\subseteq \gr_NkG$ acts on $X\otimes Y$ through its action on $X$ only.
 \end{lem}
 In the display, $gu_n\in (\rad kN)^ng\subseteq (\rad kN)^nG$ and $[gu_n]_n:= gu_n + (\rad kN)^{n+1}\in X_n$, $g\in G$.
 
 \begin{proof}We will  define the require action of $\gr_NkG$ on $X\otimes Y$ by means of a diagram

 \medskip\medskip 
 \begin{tikzpicture}[x=10pt,y=10pt,>=stealth,scale=1.25]
\draw
(-11,4) node {$kG\otimes {\rm gr}_N(kN) \otimes X \otimes Y$}
(0,4) node {$kG\otimes X \otimes Y$}
(8,4) node {$X \otimes Y$}
 
(-11,-4) node {${\rm gr}_N(kG) \otimes X \otimes Y$}
 
(-4.5,4.75) node {$\theta$}
(4.5,4.75) node {$\tau$}
 
(.5,-1) node {$\delta$}
(-11.5,0) node {$\gamma$}
;
\draw[thick,->] (-6,4) -- (-3,4);
\draw[thick,->] (3,4) -- (6,4);
 
\draw[thick,->] (-11,3) -- (-11,-3);
\draw[thick,->] (-11,3) -- (-11,-2.5);
 
\draw[thick,dashed,->] (-7,-3.5) -- (7,3);
\end{tikzpicture} 
 \medskip\medskip
  
 In this diagram, $\theta:kG\otimes\gr_NkN\otimes X\otimes Y\to kG\otimes X\otimes Y$ is defined by $a\otimes b\otimes x\otimes y
 \mapsto a\otimes bx\otimes y$, $a\in kG, b\in \gr_NkN$, $x\in X, y\in Y$, using the given action of $\gr_NkN$ on $X$. On the other hand,
 $\tau:kG\otimes X\otimes Y\to X\otimes Y$ is defined by $\tau(g\otimes x\otimes y):= gx\otimes gy$, $g\in G, x\in X, y\in Y$. 
 Finally, $\gamma$, a mapping of graded vector space, is defined by multiplication on each grade, regarding $kG$ in the top left as concentrated in grade 0 and using multiplication on $\gr_NkN$.  Also, there are maps  
 $$kG\otimes\gr_NkN\otimes X\otimes Y\overset{\beta}\longrightarrow kG\otimes_{kN}\gr_NkN\overset{\alpha}\longrightarrow 
 \gr_NkG,$$ 
 where $\gr_NkN$ is regarded as a $kN$-module, grade by grade. Thus, $\gamma=\gamma'\otimes 1_X\otimes 1_Y$, if $\gamma':=\beta\circ\alpha$. But $\dim kG\otimes_{kN}\gr_NkN
 =\dim \gr_NkG$. Thus, $\alpha$ is an isomorphism. Hence, $\Ker(\gamma)=\Ker(b)$. But $\Ker(\beta)$
 is the $k$-space generated by expressions $an\otimes b- a\otimes nb$. By definition, these
 elements are killed by $\theta\circ\theta$, so there exists a unique $\delta$ making the above
 diagram commute. In this way, $X\otimes Y$ becomes a $\gr_NkG$-module as required.
 \end{proof}

Continue to assume the hypotheses and notation of Theorem \ref{Generalizationtofinitegroups}. Let
$\sC=\sC(L)$ the full abelian subcategory of $kG$-mod consisting of finite dimensional $kG$-modules $M$ with the
property that the restriction $M_N$ of $M$ to $N$ is a direct sum of copies of $L$. The following result
is a special case of \cite[Thm. 3.11]{Cline}; a similar idea in the context of algebraic groups was noted
in \cite{CPS1} and was discovered independently by Jantzen \cite[2.2(1)]{Jan0}; see also \cite[I.6.15(2)]{Jan}. For the convenience of the reader, we include a proof.

\begin{lem}\label{Cline} Assume that $L$ is absolutely\footnote{This assumption is largely
a convenience. Without this assumption, it is still true that, if $E:=\End_{kN}(L)^{\text{\rm op}}$, then 
the functor $M\mapsto \Hom_{kN}(L,M_N)$ gives an equivalence of $\sC$ with $E(G/N)$-mod, where the
algebra $E(G/N)$ is a ``twisted"
group ring (via the natural action of $G/N$ on $E$. The twisting is through the natural action of
$G/N$ on $E$.  In fact, an inverse equivalence is given by $X\mapsto L\otimes_EX$; see Cline's work
\cite[Thm. 3.11]{Cline}, which gives even more general versions.} irreducible. The category $\sC$
is equivalent to $kG/N$-mod by means of the functor $M\mapsto \Hom_N(L,M_N)$. In addition,
$M\cong L\otimes \Hom_N(L,M_N)$ in $kG$-mod.
\end{lem}

\begin{proof} First, if $f\in\Hom_N(L_N,M_N)$ and $g\in G$, put $(g\cdot f)(x):= g\cdot f(g^{-1}\cdot x)$
for all $x\in L$. This defines an action of $G$ on $\Hom_N(L_N,M_N)$, and since $g\cdot f=f$
for all $n\in N$, $\Hom_N(L_N,M_N)$ is a $kG$-module. Thus, $\Hom_N(L_N,-):\sC(N)\to kG/N$-mod
is an exact additive functor. There is a natural $\sC(N)$-homomorphism $\Phi(M):L\otimes \Hom_N(L_N,M_N)
\to M$, $v\otimes f\mapsto f(v)$. If $M$ is irreducible, then $\Phi(M)$ is surjective, and hence an
isomorphism by dimension considerations. Now the exactness of $\Hom_N(L_N,-)$ implies easily
that $\Phi(M)$ is an isomorphism generally. Clearly, $\Phi$ provides an inverse to $\Hom_N(L_N,-)$.
\end{proof}

We can now establish the following result.

\begin{thm}\label{mainfinitetheorem}Assume the hypotheses and notation of Theorem \ref{Generalizationtofinitegroups}. Let $S$ be an irreducible $kG/N$-module with
projective cover $\widetilde Q$ in $kG/N$-mod.  Let $Q$ be the projective cover of the $kG$-module $L\otimes S$. Then $(\gr_NT)\otimes \widetilde Q$ becomes a naturally graded $\gr_NkG$-module, and, as such, is
isomorphic to $\gr_NQ$. \end{thm}

\begin{proof} First, $M:=Q/(\rad kN)Q$ is in $\sC$. To see this, suppose $L'$ is a $kN$-irreducible
summand of the completely reducible $kN$-module $Q$. Then $\Hom_{kG}(Q, \ind_N^GL')\not=0$. Thus, the
irreducible $kG$-module
$L\otimes S$ appears as a composition factor of $\ind_N^GL'$. Restricting back  to $N$, it follows that
$L$ is a composition factor of a ``twist" $L^{\prime g}$ of $L'$ by some $g\in G$.  But $L$ is $G$-stable, so that $L'\cong L$ in $kN$-mod, and hence 
$M\in \sC$. 

Next, it follows directly (from the projectivity of $Q$) that $M$ is  projective in $\sC$.  Of course, $M\not=0$
since $\rad kN$ is nilpotent. (Indeed, $(\rad kN)kG$ is a nilpotent ideal in $kG$, as is $\rad(\gr_NkN)
\gr_NkG$ in $\gr_NkG$---this will be useful below---and so contained in
$\rad kG$.) Also, $M$ is a $G$-homomorphic image of $Q$, which has $kG$-head
$L\otimes S$. Thus, $M$ has $kG$-head $L\otimes S$, and so is indecomposable in both $kG$-mod
and in $\sC$. Indeed, $M$ is the projective cover in $\sC$ of $L\otimes S$.
  By Lemma 4.3, $M\cong L\otimes \wQ$
in $\sC$ (and also in $kG$-mod). 
 
The $kN$-head of $Q$ (which is $M$ by construction) is isomorphic to the $\gr_NkN$-head of $\gr_NQ$ as a 
$kG/(\rad kN)kG\cong \gr_NkG/\rad\gr_NkN(\gr_NkG)$-module. (To see this isomorphism, observe that it holds
when $Q$ is replaced
by $kG$.)
That is, $M\cong L\otimes\wQ$ is the $\gr_NkN$-head of $\gr_NQ$, the latter module a projective $\gr_NkG$-module. 
Applying the last assertion in Lemma 4.2,  $L\otimes \wQ$ is also the $\gr_NkN$-head  of $\gr_NT\otimes \wQ$. In particular, the $\gr_{N}(kG)$-head of $\gr_NT\otimes \wQ$ is the same as that of $\gr_NQ$. Therefore, there is a surjection
$\pi:\gr_NQ\twoheadrightarrow \gr_NT\otimes \wQ$ in $\gr_NkG$-mod. Also, $\gr_NQ$ is a projective $\gr_NkN$-module.

By Lemma 4.2 again and construction,  $\gr_NT\otimes\wQ$ is projective in the
category of $\gr_NkN$-modules. Also, $\gr_NT\otimes \wQ$ and $\gr_NQ$ have the
same $\gr_NkN$-head. Hence, they are isomorphic as $\gr_NkN$-modules, and, in particular, they
have the same dimension. Therefore, the surjection from $\pi$ above is an isomorphism of $\gr_NkG$-modules, as required. \end{proof}

The following corollary generalizes \cite[Theorem]{ACS}.

\begin{cor}\label{ACScor} Let $N$ be a normal subgroup of a finite group $G$. Fix a field $k$ and let $L$  be an absolutely
irreducible $kN$-module which extends to the group $G$. (For example, $L$ could be the trivial module $k$.) Then there exists a
finite dimensional $kG$-module $M$ with the following property:

 Given any irreducible $kG/N$-module 
$S$, let $Q$ be the projective cover in $kG$-mod of $L\otimes S$, and let $\wQ$ be the projective cover 
in $kG/N$-mod of $S$.  Then $Q$ and $M\otimes\wQ$ have the same $G$-composition factors (counting
multiplicities).  In fact, $M$ can be chosen so that $M_N$ is completely reducible, and there
is a $kG$-module isomorphism
$$\bigoplus_{n\geq 0}\frac{(\rad kN)^nQ}{(\rad kN)^{n+1}Q}\cong M\otimes \wQ.$$
  \end{cor}

\begin{proof} We apply Theorem \ref{mainfinitetheorem}.  Since $\gr_NT$
is a $\gr_NkG$-module, each 
section $(\gr_NT)_n$ can be be naturally regarded as a $kG$-module though the action of
$(\gr_NkG)_0=kG/(\rad kN)kG$. Let  $M:=\bigoplus_n(\gr_NT)_n$.\end{proof}

\begin{rem}\label{lastremark}(a) As the proof of Corollary \ref{ACScor} shows, $M$ can be chosen so that $M_N$
is completely reducible, and isomorphic to the direct sum (with multiplicities) of all composition factors
of the projective cover $T$ in $kN$-mod of $L$.

(b) In case $L=k$, all of the additional properties (beyond the statement of \cite[Theorem 1]{ACS}) can be
deduced by slightly extending the proof given in \cite{ACS}. There $M$ is taken to be a direct sum of
the $kG$-modules $J^n/J^nA$, where $J=\rad kN$ and $A$ is the augmentation ideal of $kN$. The
action of $G$ is by conjugation. Observe that this conjugation action of $G$ becomes just the left
action, upon restriction to $N$. Also, note that $J^n/J^nA$ is a quotient of $J^n/J^{n+1}$ as a 
$(kN,kN)$-bimodule since $A\supseteq J$. Moreover, the quotient map $J^n/J^{n+1}\twoheadrightarrow
J^n/J^nA$ is split as a bimodule map, since $J^n/J^{n+1}A=(J^n/J^{n+1})A$ is, as a right $kN$-module
the sum of all non-trivial $kN$-irreducible right submodules of $J^n/J^{n+1}$. Now write $1=\sum_{i\geq 0} e_i$, as
a sum of primitive orthogonal idempotents in $kN$. Assume $e=e_0$ is such that $kNe$ is the PIM of
the trivial module $k$. Equivalently, $e_0\in A$. Thus, $ekN$ is the PIM of the right trivial module $k$. 
Taken with the discussion above, this gives that a $(kN,kN)$-bimodule decomposition $J^n/J^{n+1}
= (J^n/J^{n+1}e\oplus J^n/J^nA$. In particular, there are left $kN$-module isomorphisms 
$$J^n/J^nA\cong (J^n/J^{n+1})e\cong J^ne/J^{n+1}e\cong J^n(kNe)/J^{n+1}(kNe).$$
This proves the additional properties of $M$ in Corollary \ref{ACScor} from the point of view of \cite{ACS}. The
proof in \cite{ACS} does give an isomorphism $J^nQ/J^{n+1}Q\cong (J^n/J^{n+1}A)\otimes \wQ$ as
$kG$-modules. But it does not identify $J^n/J^{n+1}A$ in terms of $kNe\cong T$.

However, we also note, the method of \cite{ACS}, with the identification $J^n/J^nA\cong(\gr_NT)_n$
in hand, gives another way to prove Theorem \ref{mainfinitetheorem} in the special case $L=k$ (as well
as Theorem 3.1 for PIMs with trivial head).

(c) Part (a) above implies that the character of the $kN$-module $T$ is the restriction of the character of
a $kG$-module. That is, the class $[T]$ in the Grothendieck group of $kN$-modules is the image
under restriction of the class $[M]$ (in the Grothendieck group of $kG$) of a $kG$-module $M$. A similar
conclusion holds in the context of Theorem 3.1 and is already already noted in Donkin \cite{Donkin1} as a main result (see Corollary, p. 149). Using Donkin's result as a starting point, character factorizations similar
to \cite[Theorem]{ACS} can be proved, using for $N$ the infinitesimal kernel of a power $F^r$ of
the Frobenius morphism, and for $G$ the infinitesimal kernel of a higher power $F^s$. We do not pursue
this further, leaving details to the interested reader. But we point out that there is a close connection
between the work, at the character level, of \cite{ACS} and \cite{Donkin1} as well as with the present paper.

(d) As suggested already, above Lemma 4.3, there are connections of this paper with the theory of
Clifford systems, as defined by Dade \cite[\S1]{Dade} and used in \cite{Cline}. Further ties are suggested by the following observation:
The algebras $\gr_NkG$ in this section are all Clifford systems over the group $G/N$. (This appears
to be a new observation.)  It seems
likely that further generalizations of \cite[Theorem]{ACS} might arise by applying the results
and methods of \cite{Cline} to this Clifford system. We do not pursue this, but leave it to the interested
reader.  

\end{rem}

 \section{The general case} In this section, let $G$ be a connected, affine algebraic group over a fixed
 algebraically closed field $k$, possibly of characteristic 0.   Recall that Theorem \ref{maintheorem} concerned strongly stable pairs $(Q,V)$ for a normal, infinitesimal subgroup
 of $G$. This section presents a considerable generalization of this result, in which the subgroup $N$
 is not assumed to be infinitesimal.  
 
 The connected assumption on $G$ is largely a convenience allowing us to quote certain results in \cite[\S3]{PS1}; in
 particular, it is not needed in the following two lemmas.

\begin{lem}\label{lemma1} Suppose that $G\to A^\times$ is a homomorphism of affine algebraic groups,
where $A^\times$ is the group of units in a finite dimensional $k$-algebra $A$. Then there is a natural
homomorphism $\Dist(G)\to A$ of $k$-algebras, such that, for any $A$-module $M$, the action of $G$ on $M$
through $A^\times$ induces the action of $\Dist(G)$ on $M$ through $A$-multiplication.\end{lem}

\begin{proof}
Let $M={_AA}$ be the (left) regular $A$-module. Then $M$ is a rational $G$-module by means of the
homomorphism $G\to A^\times$, and so it is a $\Dist(G)$-module by means of an algebra homomorphism
$\Dist(\rho):\Dist(G)\to\End(M)$. For $a\in A$, right multiplication operator $a_R$ on $M$ commutes with the action of $A$, and
hence of $G$, on $M$. Thus, $a_R$ commutes with the action of $\Dist(G)$ on $M$. Since $A\cong \End_A({A_A})\cong A$, $\Dist(\rho)$ can be viewed as an algebra homomorphism $\Dist(G)\to A$. With this identification,
the action of $\Dist(G)$ on $M$ is induced through a natural algebra homomorphism $\Dist(G)\to A$.

Similarly,  $A$-multiplication and the map $\Dist(G)\to A$ give the action induced by $G$
on any finite direct sum $A^{\oplus n}$ of copies of $A$, or any of its submodules, or quotient modules
(for the action of $A$). The lemma now follows.\end{proof}

\begin{lem}\label{lemma2} Let $G$ be an affine algebraic group over $k$. Let $M$ be any vector
subspace of finite codimension in $\Dist(G)$. Then $M$ contains a two-sided ideal $I$ of $\Dist(G)$
of finite codimension in $\Dist(G)$. \end{lem}

\begin{proof}Let $F\subseteq k[G]$ be the annihilator of $M$. Then the space of linear maps $F\to k$
which are restrictions from $k[G]$ to $F$ of linear functionals in $\Dist(G)$ is finite dimensional. Hence, $F$ is finite dimensional,
since the intersections of all powers of the maximal ideal at 1 of $k[G]$ is $0$.

Since $k[G]$ is a rational $G\times G$-modules (as induced by the action $\{G\times G\}\times G\to G$, $(g_1,g_2,g)\mapsto
g_1gg_2^{-1}$, there is a $G\times G$-stable submodule $\widehat F$ containing $F$ as a subspace, and
which is finite dimensional. Let $I$ be the annihilator of $\widehat F$ in $\Dist(G)$. If $x\in I$ and $y\in
\Dist(G)$, and $f\in\widehat F\subseteq k[G]$, then
$$\begin{cases} (xy)(f)=\sum x(f_1)y(f_2)=0,\\
 (yx)(f)=\sum y(f_1)x(f_2)=0,
\end{cases}$$
where $f\mapsto \sum f_1\otimes f_2$, in Sweedler notation, under the comultiplication $k[G]
\to k[G]\otimes  k[G]$. Thus, $I$ is a two-sided ideal, and the lemma follows.
\end{proof}

 Let $N$ be a closed, normal subgroup scheme of $G$. For simplicity, we assume that $N$ is
 connected. Let $(Q,V)$ a strongly $G$-stable pair as 
 defined in \S2.

\begin{thm}\label{footnotetheorem}Let $(Q,V)$ be a strongly $G$-stable pair as above. Let $B'=\Dist(G)/I'$ be any
finite dimensional quotient algebra of $\Dist(G)$ by an ideal $I'$. There exists a finite dimensional quotient algebra $B=\Dist(G)/I$ with $I\subseteq I'$ having the following properties:

(a) If $\fa$ denotes the image of $\Dist(N)\subseteq\Dist(G)$ in $B$, then the action of $\Dist(N)$ on
$Q$ factors through an action of $\fa$ on $Q$. Moreover, the action of $\gr_\fa \fa$ on $\gr_\fa Q$ extends 
to an action of $\gr_\fa B$ on $Q$.

(b) In addition, the actions in (a) can be chosen so that 
 $$\gr^\#_\fa V:=\bigoplus_{n\geq 0} \frac{V \cap(\rad\fa)^nQ}{V\cap(\rad\fa)^{n+1}Q}$$
 is a $\gr_\fa B$-submodule of $\gr_\fa Q$, in which the action of $G$ on each $(\gr^\#_\fa V)_n$ agrees with the
 action of $\Dist(G)$ through $B/(\rad\fa)B=(\gr_\fa B)_0$. 
 
 (c) If $B^{\prime\prime}=\Dist(G)/I^{\prime\prime}$ is any finite dimensional quotient algebra of $\Dist (G)$
 by an ideal $I^{\prime\prime}\subseteq I$, then both (a) and (b) above hold with $B$ replaced by $B^{\prime\prime}$.
\end{thm}

\begin{proof}Recall the 
 construction of the group $G^\diamond$ described in \S2. The factorization $G^\diamond=(U,1)\times
(1,G)$ as $k$-schemes (varieties, in fact) provides a tensor decomposition
$$k[G^\diamond]\cong k[U]\otimes\ k[G]$$
of coordinate algebras, 
and a vector space decomposition
$$\Dist(G^\diamond)\cong \Dist(U)\otimes\Dist(1,G),$$
compatible with the multiplication of subspaces in $\Dist(G^\diamond)$.
Here we have identified $\Dist(U)$ with a subalgebra, stable under the adjoint action, of $\Dist(G^\diamond)$.
The subspace $\Dist(1,G)$ of $\Dist(G^\diamond)$ is the image of $\Dist(G) =\Dist(1,G)$ in
$\Dist(1,G^\diamond)=\Dist(G^\diamond)$, using the procedure of [Jan; I,\S7] of defining distribution
algebras for arbitrary $k$-schmes equipped with a distinguished rational point. We keep the notation
$\Dist(1,G)$ as a reminder that $\Dist(1,G)\subseteq\Dist(G^\diamond)$ is only a subspace, not
a subalgebra of $\Dist(G^\diamond)$. The subspace $\Dist(1,G)$ contains the identity element of $\Dist(G^\diamond)$ and
is stable under left and right multiplication by $\Dist(N)$, identifying the latter algebra with its isomorphic
copy $\Dist(\iota(N))\subseteq\Dist(G^\diamond)$. Finally, under the natural map $\Dist(G^\diamond)\to
\Dist(G)$ of algebras, the subspace $\Dist(1,G)$ maps bijectively onto $\Dist(G)$.

Consider the algebra homomorphism $\phi:\Dist(G^\diamond)\to\End_k(Q)$ which induced by the affine
algebraic group map $G^\diamond\to GL_k(Q)$, which is the identity on
$U=1_Q+J_V$, which is contained in the unit group of the algebra $k+J_V$. Thus, $\Dist(U)$ has image
contained in (and, so, equal to) $k+J_V$. The composite algebra homomorphism $\Dist(U)\to k+J_V\to k$ is the
augmentation homomorphism of $\Dist(U)$. (To see this, note that $V$ is a trivial $U$-submodule of $A$, hence is a trivial $\Dist(U)$-submodule of $Q$. Thus, multiplication of $\Dist(U)$ on $V$ factors through the augmentation map $\Dist(U)\to k$. On the other hand, multiplication of $\Dist(U)$ on $Q$ factors through the
map $\Dist(U)\to k+J_V$ by Lemma \ref{lemma1}. If the action is restricted to $V$, $J_V$ kills $V$, so that
the restriction map $k+J_V\to\End_k(V)$ factors through $k+J_V\to k$. That is, the action of
$\Dist(U)$ on $V$, which we know to be trivial, is through $\Dist(U)\to k+J_V\to k$, and so the latter composition map is the augmentation homomorphism). As a consequence, the map $\Dist(U)\to k+J_V$ sends
$\Dist^+(U)$ to $J_V$. That is, $\phi(\Dist^+(U))=J_V$.

We next consider the restriction $\phi|_{\Dist(1,G)}$ and its composite $\phi_G=\phi|_{\Dist(1,G)}\circ
\iota_G$ with the isomorphism $\iota_G:\Dist(G)\cong\Dist(1,G)$. The latter map is both an isomorphism
of vector spaces and of $\Dist(N)$-bimodules. Since the vector space map $\phi_G:\Dist(G)\to\End_k(Q)$
has a finite dimensional image, its kernel $\Ker\phi_G$ has finite codimension in $\Dist(G)$. By Lemma
\ref{lemma2}, $\Ker\phi_G$ contains a two-sided ideal $I$ of $\Dist(G)$ of finite codimension in $\Dist(G)$.

We let $B$ denote the finite dimensional algebra $\Dist(G)/I$. Replacing the ideal $I$ by the ideal $I\cap I'$,
we can assume that $I\subseteq I'$. (The ideal $I$ could be replaced by any ideal contained in it and having
finite codimenion.)

For brevity, we write
\begin{itemize}
\item[(1)] $D_1=\Dist(U)$;
\item[(2)] $D_2=\Dist(1,G)$;
\item[(3)] $I_2=\iota_G(I)$, where, by abuse of notation, $\iota_G:\Dist(G)\overset\sim\to \Dist(1,G)$, a subspace of $\Dist(1,G)$.
\item[(4)]  $B_2=D_2/I_2 \,\,(\cong \Dist(G)/I=B$), an isomorphism of vector spaces.
\end{itemize}

Recall the algebra $\fa=(\Dist(N)+I)/I$ defined in the theorem.  Let $r(N)=r_I(N)$ denote the inverse image of $\rad\fa$
under the natural map $\Dist(N)\twoheadrightarrow\fa$. The latter map is equivariant with respect to
$G$-conjugation. Consequently, $r(N)$ is stable under $G$-conjugation, and so is stable under the
adjoint action of $\Dist(G)$, in either its left or right hand version. It then follows easily, from Hopf algebra
calculations, that
$$r(N)\Dist(G)=\Dist(G)r(N).$$
We leave the verification of this fact to the reader. Since the $\Dist(N)$-bimodule $D_2$ is isomorphic to
$\Dist(G)$, we also have
$$r(N)D_2=D_2r(N).$$
Similarly, we have identities $(\rad\fa)B=B(\rad\fa)$ and $(\rad\fa)B_2=B_2(\rad\fa)$. Finally, viewing
$D_1\otimes D_2$ and $D_1\otimes B_2$ as $\Dist(N)$-bimodules, and noting that multiplication by
$\Dist(N)$ and $\Dist(U)$ commute (elementwise), we have
$$r(N)(D_1\otimes D_2)=D_1\otimes r(N)D_2=D_1\otimes D_2r(N)=(D_1\otimes D_2)r(N)$$
and
$$(\rad\fa)(D_1\otimes B_2)=(D_1\otimes (\rad\fa)B_2=D_1\otimes B_2(\rad\fa)=(D_1\otimes B_2)\rad\fa.$$
In each instance,  the action of $\rad\fa$ in the lower equation is, by definition, induced from
the corresponding action of $r(N)$ in the upper equations. These all give well-defined actions of $\rad \fa$,
since the left and right actions of $\rad\fa$ on $B_2$ are both induced from corresponding $r(N)$-actions.

For any $\Dist(N)$-module $M$, possibly infinite dimensional, we define
$$\gr_NM=\gr_{N,r(N)}M:=\bigoplus_{i\geq 0}\frac{r(N)^iM}{r(N)^{i+1}M}.$$
Then $\gr_NM$ is a positively graded vector space. If $\psi:M\to M^{\prime\prime}$ is
a morphism of $\Dist(N)$-modules, then there is a natural homomorphism $\gr_N\psi:\gr_NM\to
\gr_NM^{\prime\prime}$ of graded vector spaces. If $\psi$ is surjective with kernel $M'$, then $\gr_N\psi$ is surjective
with kernel $\gr^\#_NM'\to\gr_NM$, where
$$\gr^\#_NM':=\bigoplus_{i\geq 0}\frac{M'\cap r(N)^iM}{M'\cap r(N)^{i+1}M}$$
maps injectively in an obvious way to $\gr_NM$.

We will now apply these constructions to the diagram
\newcommand{\ep}{\end{picture}}
\newcommand{\bp}{\begin{picture}}
\newcommand{\script}{\scriptstyle}

$$
\begin{array}{ccccccccc}
0 \hspace*{-6pt}
& \bp(20,0)\put(0,3){\vector(1,0){20}}\ep
& D_1^+ \otimes D_2
& \bp(20,0)\put(0,3){\vector(1,0){20}}\put(6,6){}\ep
& \mbox{Dist}(G^\diamond)
& \bp(20,0)\put(0,3){\vector(1,0){15}}\put(5,3){\vector(1,0){15}}\put(6,6){$\script\alpha$}\ep
& \mbox{Dist}(G)
& \bp(20,0)\put(0,3){\vector(1,0){20}}\ep
& \hspace*{-6pt} 0 \\[1mm]
&&&& \bp(5,20)\put(0,20){\vector(0,-1){20}}\put(0,20){\vector(0,-1){15}}\put(4,9){$\script\delta$}\ep
&& \bp(5,20)\put(0,20){\vector(0,-1){20}}\put(0,20){\vector(0,-1){15}}\put(4,9){$\script\beta$}\ep
&& \\[1mm]
0 \hspace*{-6pt}
& \bp(20,0)\put(0,3){\vector(1,0){20}}\ep
& D_1^+ \otimes B_2
&  \bp(20,0)\put(0,3){\vector(1,0){20}}\put(6,6){}\ep& D_1 \otimes B_2
& \bp(20,0)\put(0,3){\vector(1,0){15}}\put(5,3){\vector(1,0){15}}\put(6,6){$\script\gamma$}\ep
& B
& \bp(20,0)\put(0,3){\vector(1,0){20}}\ep
& \hspace*{-6pt} 0 \\[1mm]
&&&& \bp(5,20)\put(0,20){\vector(0,-1){20}}\put(0,20){\vector(0,-1){15}}\put(4,9){$\script\bar{\phi}$}\ep
&& \\[1mm]
&&&& \phi(\mbox{Dist}(G^\diamond))
&&
\end{array}  $$
where $D_1^+=\Dist^+(U)$, the middle vertical map is the natural factorization of $\phi$ (from
$\phi(D_1\otimes I_2)=\phi(D_1)\phi(I_2)=0$), and the square in the upper right is commutative. All
objects in the diagram are (at least) left $\Dist(N)$-modules, and all maps are (at least) $\Dist(N)$-module
homomorphisms.

Objects $\Dist(G^\diamond)=D_1\otimes D_2$, $\Dist(G)$ and $B$ are algebras, as are
$\gr_N\Dist(G^\diamond)$, $\gr_N\Dist(G)$, and $\gr_NB$. This is a consequence of the
commuting properties of $r(N)$ discussed above. Both of the maps $\alpha$ and $\beta$ are algebra
homomorphisms, as are the graded maps $\gr_N\alpha$ and $\gr_N\beta$.

The commutative square $\beta\alpha=\gamma\delta$ remains commutative after applying $\gr_N$ to each
term, so there is an induced map
$$\tau:\Ker\gr_N\alpha\to\Ker\gr_N\gamma.$$
In grade $i\in\mathbb N$, 
$$\tau:\frac{D_1^+\otimes D_2\cap(D_1\otimes r(N)^iD_2)}
{D_1^+\otimes D_2\cap (D_1\otimes r(N)^{i+1}D_2)}\longrightarrow
\frac{D_1^+\otimes B_2\cap (D_1\otimes r(N)^iB_2)}{D_1^+\otimes B_2\cap (D_1\otimes r(N)^{i+1}B_2)}$$
in the obvious way, using the surjection $D_2\twoheadrightarrow B_2$. However,
$$\begin{cases} D_1^+\otimes D_2\cap(D_1\otimes r(N)^iD_2)=D_1^+\otimes r(N)^iD_2\\
D_1^+\otimes B_2\cap (D_1\otimes r(N)^iB_2)=D_1^+\otimes r(N)^iB_2.\end{cases}$$
Consequently, $\tau$ is surjective.

Put $Y=\gr_N\phi(\Dist(G^\diamond)=\gr_N\bar\phi(D_1\otimes B_2)$, and let $X$ be the image
in $Y$ of $\Ker\gr_N\gamma$ under the composite of the maps $\Ker\gr_N\gamma\to\gr_N(D_1\otimes B_2)$
and $\gr_N\bar\phi:\gr_N(D_1\otimes B_2)\to Y$, and let $Z=X/Y$ be the graded quotient space. We have a big commutative diagram with exact rows, and with all vertical maps surjective:

$$
\begin{array}{ccccccccc}
0 \hspace*{-6pt}
& \bp(20,0)\put(0,3){\vector(1,0){20}}\ep
& \mbox{gr}^{\#}_N(D^+_1 \otimes D_2)
& \bp(20,0)\put(0,3){\vector(1,0){20}}\ep
& \mbox{gr}_N\mbox{Dist}(G^\diamond)
& \bp(20,0)\put(0,3){\vector(1,0){20}}\put(1,7){$\script{\rm gr}_{\!N}\alpha$}\ep
& \mbox{gr}_N\mbox{Dist}(G)
& \bp(20,0)\put(0,3){\vector(1,0){20}}\ep
& \hspace*{-6pt}0 \\[1mm]
&&
\bp(5,20)\put(0,20){\vector(0,-1){20}}\put(-7,9){$\script\tau$}\ep
&& \bp(5,20)\put(0,20){\vector(0,-1){20}}\put(0,20){\vector(0,-1){15}}\put(4,9){$\script{\rm gr}_{\!N}\delta$}\ep
&& \bp(5,20)\put(0,20){\vector(0,-1){20}}\put(0,20){\vector(0,-1){15}}\put(4,9){$\script{\rm gr}_{\!N}\beta$}\ep
&& \\[1mm]
0 \hspace*{-6pt}
& \bp(20,0)\put(0,3){\vector(1,0){20}}\ep
& \mbox{gr}^{\#}_N(D_1^+ \otimes B_2)
& \bp(20,0)\put(0,3){\vector(1,0){20}}\ep
& \mbox{gr}_N D_1 \otimes B_2
& \bp(20,0)\put(0,3){\vector(1,0){20}}\put(1,9){$\script{\rm gr}_{\!N}\gamma$}\ep
& \mbox{gr}_NB
& \bp(20,0)\put(0,3){\vector(1,0){20}}\ep
& \hspace*{-6pt} 0 \\[1mm]
&& \bp(5,20)\put(0,20){\vector(0,-1){20}}\put(0,20){\vector(0,-1){15}}\ep
&& \bp(5,20)\put(0,20){\vector(0,-1){20}}\put(0,20){\vector(0,-1){15}}\put(4,9){$\script{\rm gr}_{\!N}\bar{\phi}$}\ep
&& \bp(5,20)\put(0,20){\vector(0,-1){20}}\ep
&& \\[1mm]
0 \hspace*{-6pt}
& \bp(20,0)\put(0,3){\vector(1,0){20}}\ep
& \hspace*{-7pt} X
&  \bp(20,0)
\put(1,3){\vector(1,0){16}}\ep
& \hspace*{-7pt} Y
& \bp(20,0)\put(0,3){\vector(1,0){20}}\ep
& \hspace*{-7pt} Z
& \bp(20,0)\put(0,3){\vector(1,0){20}}\ep
& \hspace*{-6pt} 0
\end{array}
$$

Since $\gr_N\alpha$ is a graded algebra homomorphism, the image in $\gr_N\Dist(G^\diamond)$ of $\gr^\#(D_1^+\otimes D_2)$ is a graded ideal. Hence, the image $X$ of this ideal under the graded algebra
surjection $\gr_N\phi:\gr_N\Dist(G^\diamond)\twoheadrightarrow\gr_N\phi(\Dist(G^\diamond))$ is also a
graded
ideal. In
particular, $Z=Y/X$ has the structure of a graded algebra. Also, the right hand vertical map $
\gr_N\Dist(G)\twoheadrightarrow Z$ is a graded algebra homomorphism. Since the surjection
$\gr_N\Dist(G)\to\gr_NB$ is a graded algebra homomorphism, it follows that $\gr_NB\to Z$ is a graded algebra
homomorphism.

To prove that $\gr_NB$ acts on $\gr_NQ$, it suffices to show that the graded algebra $Z$ acts on
$\gr_NQ$. Then $\gr_NB$ will acts through the (surjective) algebra homomorphism $\gr_NB\to
X$. However, the algebra $Y$ already acts on $\gr_NQ$, since $\gr_N\bar\phi\circ \gr_N\delta=\gr_N\phi$.
Thus, it suffices to show that $X$ acts trivially on $\gr_NQ$.

The graded ideal $X$ may be computed from its definition as an image of $\Ker\gr_N\gamma$ in $Y$,
and from the discussion of the individual grades of that kernel. For $i\in\mathbb N$,
$$X_i=\frac{J_V(\rad\fa)^i\bar\phi(B_2)+(\rad\fa)^{i+1}\phi(\Dist(G^\diamond)}{(\rad\fa)^{i+1}\phi(\Dist(G^\diamond)},$$
which acts trivially on $\gr_NQ$. Thus, the module $\gr_NQ$ becomes $\gr_NB$-module.

Statement (a) follows immediately. The first statement in (b) is obvious. The action of $G$ on $V$
comes from the action of $G^\diamond$. For $\gr_NV$, the action of $J_V$ is trivial, so the second
statement in (b) holds.
\end{proof}

\end{document}